\documentclass[sn-mathphys,Numbered]{sn-jnl}

\usepackage{amsmath,amssymb,amsfonts}%
\usepackage{amsthm}%
\usepackage{mathrsfs}%
\usepackage[title]{appendix}%
\usepackage{xcolor}%
\usepackage{textcomp}%
\usepackage{manyfoot}%
\usepackage{booktabs}%
\usepackage{orcidlink}
\usepackage[utf8]{inputenc}

\usepackage{enumitem}
\usepackage{mathrsfs}
\usepackage{mathtools}

\setlist[enumerate]{label=\emph{(\roman*)}}

\usepackage{environ}

\newtheorem{theorem}{Theorem}[section]

\newtheorem{lemma}[theorem]{Lemma}
\newtheorem{proposition}[theorem]{Proposition}

\theoremstyle{definition}
\newtheorem{definition}[theorem]{Definition}
\newtheorem{remark}[theorem]{Remark}
\numberwithin{equation}{section}
\newcommand{\R}{\mathbb{R}}

\def \C {{\mathbb{C}}}
\def \d {{\rm{d}}}
\def \pt{\partial_{t}}

\begin{document}
	
	\title[Uniform-in-mass global solutions to 4D DKG systems]{Uniform-in-mass global existence for 4D Dirac-Klein-Gordon equations\footnote{This preprint has not undergone peer review (when 
			applicable) or any post-submission improvements or corrections. The Version of Record of this article is published in Analysis and Mathematical Physics, and is available online at https://doi.org/10.1007/s13324-024-00945-8.}}

	
	\author[1]{\fnm{Jingya} \sur{Zhao} \orcidlink{0009-0001-1090-9651}
	}\email{12231282@mail.sustech.edu.cn}
	
	\affil[1]{\orgdiv{Department of Mathematics}, \orgname{Southern University of Science and Technology}, \orgaddress{\city{Shenzhen}, \postcode{518055}, \country{China}}}

	
	\abstract{We are interested in four-dimensional Dirac-Klein-Gordon equations, a fundamental model in particle physics. The main goal of this paper is to establish global existence of solutions to the coupled system and to explore their long-time behavior. The results are valid uniformly for mass parameters varying in the interval $[0,1]$. }

	
	

	\maketitle
	
	\section{Introduction}

	\subsection{Model problem}
	
	The Dirac-Klein-Gordon system is a fundamental model in particle physics, which describes a scalar Klein-Gordon field and a Dirac field through Yukawa interactions. In this paper, we study the four-dimensional Dirac-Klein-Gordon system, which is expressed as
	\begin{equation}\label{equ:DKG}
		\left\{\begin{aligned}
			-i\gamma^{\mu}\partial_{\mu}\psi+M\psi&=vF\psi,\quad \quad \ \ (t,x)\in [t_0,\infty)\times \R^{4},\\
			-\Box v+m^2v&=\psi^{*}H\psi,\quad\quad (t,x)\in [t_0,\infty)\times \R^{4},
		\end{aligned}\right.
	\end{equation}
	where $M,m\in[0,1]$ denote the masses of the Dirac field $\psi$ and the Klein-Gordon field $v$, respectively.

	In the above, 
	$\left\{\gamma^{0},\gamma^{1},\gamma^{2},\gamma^{3},\gamma^{4}\right\}$ are Dirac matrices satisfying the identities
	\begin{equation}\label{equ:gamma}
		\left\{\gamma^{\mu},\gamma^{\nu} \right\}:=\gamma^\mu \gamma^\nu + \gamma^\nu \gamma^\mu = -2\eta_{\mu\nu}I_{4},\quad 
		(\gamma^{\mu})^{*}=-\eta_{\mu\nu}\gamma^{\nu},
	\end{equation}
	in which $A^{*}$ is the conjugate transpose of a matrix $A$, $I_{4}$ is the $4\times4$ identity matrix, and $\eta=\rm{diag}(-1,1,1,1,1)$ denotes the Minkowski metric in $\R^{1+4}$. 
	There exist matrices of size $4\times4$ satisfying~(\ref{equ:gamma}). For example, the Dirac matrices can be chosen as
	\begin{equation*}
		\gamma^{0}=\begin{pmatrix}
			1 & 0&0&0\\
			0&1&0&0\\
			0&0&-1&0\\
			0&0&0& -1
		\end{pmatrix}, \quad 
		\gamma^{1}=\begin{pmatrix}
			0 & 0&0&1\\
			0&0&1&0\\
			0&-1&0&0\\
			-1 & 0&0&0
		\end{pmatrix},\quad
		\gamma^{2}=\begin{pmatrix}
			0&0&0 & -i\\
			0&0&i&0\\
			0&i&0&0\\
			-i & 0&0&0
		\end{pmatrix},
	\end{equation*}
	\begin{equation*}
		\gamma^{3}=\begin{pmatrix}
			0&0&1 & 0\\
			0&0&0&-1\\
			-1&0&0&0\\
			0& 1&0&0
		\end{pmatrix}\quad  \mbox{and}\quad
		\gamma^{4}=-\gamma^{0}\gamma^{1}\gamma^{2}\gamma^{3}=\begin{pmatrix}
			0&0&i & 0\\
			0&0&0&i\\
			i&0&0&0\\
			0 & i&0&0
		\end{pmatrix}.
	\end{equation*}
	The matrices $F,H$ are of size $4\times 4$, and following the early works~\cite{Bache,DongWyatt}, they are assumed to satisfy
	\begin{equation}\label{FH}
		(\gamma^0F)^*=\gamma^0F, \quad H^*=H.
	\end{equation}
	In (\ref{equ:DKG}), $\Box=\eta^{\mu\nu}\partial_\mu\partial_\nu=-\partial_t\partial_t+\Delta$ is the D'Alembert operator.
	The Greek indices $\mu, \nu, \cdots \in\left\{0,1,2,3,4\right\}$, and the Einstein summation convention for repeated upper and lower indices is adopted.
	
	The initial data are prescribed on the slice $t=t_0$ (we will always take $t_0=0$)
	\begin{equation}\label{equ:initial}
		\left(\psi,v,\pt v\right)|_{t=t_0}=\left(\psi_{0},v_{0},v_{1}\right),
	\end{equation}
	in which $(\psi_{0},v_{0},v_{1}):\mathbb{R}^4 \to \mathbb{C}^4\times \R\times \R$.
	
	We are interested in the Cauchy problem (\ref{equ:DKG})-(\ref{equ:initial}) under the conditions~(\ref{FH}).
	
	\subsection{Main result}
	
	The main goal of this paper is to establish global existence of solutions to the system (\ref{equ:DKG})-(\ref{equ:initial}), which is uniform in the mass parameters $M,m\in [0,1]$. Our main result is the following.
	\begin{theorem}\label{thm:main}
		Let $N\in \mathbb{N}$ with $N\geq7$. There exists an $\varepsilon_0>0$ independent of mass parameters $M,m\in[0,1]$ such that for all initial data $\left(\psi_{0},v_{0},v_{1}\right)$ 
		satisfying the smallness condition
		\begin{equation}\label{est:smallness}
			\begin{aligned}
				\sum_{0\leq k\leq N}\|\langle |x|\rangle^N\nabla^k\psi_{0}\|_{L^2_x}
				+\sum_{0\leq k\leq N+1}\|\langle |x|\rangle^N\nabla^kv_0\|_{L^1_x}+\sum_{0\leq k\leq N+1}\|\langle |x|\rangle^N\nabla^kv_0\|_{L^2_x}\\
				+\sum_{0\leq k\leq N}\|\langle |x|\rangle^N\nabla^kv_1\|_{L^1_x}+\sum_{0\leq k\leq N}\|\langle |x|\rangle^N\nabla^kv_1\|_{L^2_x}\le \varepsilon<\varepsilon_0,
			\end{aligned}
		\end{equation}
		the Cauchy problem~\eqref{equ:DKG}-\eqref{equ:initial} admits a global-in-time solution $(\psi,v)$, which enjoys the following time decay estimates
		\begin{equation}\label{est:thmpoint}
			\left|\psi(t,x)\right|\lesssim \varepsilon \langle t\rangle^{-\frac{3}{2}},\quad 
			\left|v(t,x)\right|\lesssim \frac{\varepsilon}{1+t^{\frac{5}{4}}+mt^{\frac{3}{2}}}.
		\end{equation}
	\end{theorem}
	
	\begin{remark}
		In Theorem~\ref{thm:main}, we establish a uniform-in-mass solution globally to the system~(\ref{equ:DKG}). However, we need to point out that the decay rates of Dirac field $\psi$ and Klein-Gordon field $v$ in (\ref{est:thmpoint}) are not very fast, which is caused by the limitations of the method. That is why we work in four space dimensions. In the future, we hope to improve our method to treat the system~(\ref{equ:DKG}) in three space dimensions and to get unified optimal decay estimates.
	\end{remark}
	
	\begin{remark}
		If $M,m\in\{0,1\}$ which denote the masses of the Dirac field $\psi$ and the Klein-Gordon field $v$, respectively, then we can obtain improved pointwise decay estimates for solutions.
	\end{remark}
	\begin{remark}
		Thanks to the faster decay rates of solutions in higher than 5 dimensions compared to the 4D case, we expect to obtain similar results  to Theorem~\ref{thm:main} by the same way.
	\end{remark}
	
	\begin{remark}
		When $F=I_4, H=\gamma^0$, the system (\ref{equ:DKG}) represents the most well-known case. For this case, there exist extensive mathematical studies; one can refer to \cite{Bache,Boura}.
	\end{remark}
	
	\begin{remark}
		According to~\cite[Theorem 4.4]{Dong-Li22} and~\cite{DLMY,DongWyatt}, the system (\ref{equ:DKG}) enjoys linear scattering.
	\end{remark}

	\subsection{References}
	
	The Dirac-Klein-Gordon system is a classical and important model in particle physics, and it also has high research interest in mathematics. Over the past few decades, there is a great deal of literature on the Dirac-Klein-Gordon system. The global existence and the long-time behavior for the Cauchy problem are the focus of the present article.
	
	\smallskip
	For the three-dimensional case, Chadam-Glassey~\cite{ChGl74} proved global existence of the solution to the Cauchy problem for the Dirac-Klein-Gordon system with some suitably chosen initial data. After that, Choquet-Bruhat~\cite{Ch-Br81} established global existence of the 3D massless Dirac-Klein-Gordon system  by the conformal method. Later on, Bachelot~\cite{Bache} showed global existence for a massive Dirac equation coupled with a massless Klein-Gordon equation in the three-dimensional case by the energy method; see also the recent work in~\cite{DLY}. Some results of low regularity are also mentioned here. D'Ancona-Foschi-Selberg~\cite{DFS-07} proved almost optimal local well-posedness for the 3D Dirac-Klein-Gordon system with the low regular initial data. Later on, Wang~\cite{WangX} and Bejenaru-Herr~\cite{Bejenaru-Herr} independently established global existence for 3D massive Dirac-Klein-Gordon equations with the low regular initial data.

	\smallskip
	
	For the two-dimensional case, Bournaveas \cite{Boura} established local-in-time low regularity solutions to Dirac-Klein-Gordon equations. Later on, Gr\"{u}nrock-Pecher \cite{GP10} first proved global existence of 2D Dirac-Klein-Gordon equations for low regularity (and also high regularity) initial data. Recently, Dong-Wyatt \cite{DongWyatt} studied the global behavior for some 2D coupled Dirac-Klein-Gordon equations with a massless Dirac field and a massive Klein-Gordon field for compactly-supported initial data; see also the work \cite{DLMY} by Dong-Li-Ma-Yuan for non-compactly supported initial data.
	
	\smallskip
	Relevant results for the 1-dimensional Dirac-Klein-Gordon system can be found in~\cite{Boura00,ChGl74}.
	
	\smallskip
	Last but not least, we very briefly review some results relevant to our study. Georgiev~\cite{Geor91} established global solutions to Maxwell-Dirac equations. Later on, Tsutsumi~\cite{Tsutsumi} considered global solutions to 3D Dirac-Proca equations. Recently, Dong-LeFloch-Wyatt~\cite{DLW} provided a uniform-in-mass decay result for a 3D Dirac equation. For other related results, one is led to~\cite{Dong,Ionescu-P-EKG,Katayama,Kata-Kubo,Klainerman-WY,ZhangQ23,ZhangQ}.

	\subsection{Major difficulties}
	
	We are interested in four-dimensional Dirac-Klein-Gordon equations when the mass parameters vary in the interval $[0,1]$. The main goal of this paper is to establish uniform-in-mass global existence and to explore the global behavior of the solution. The first thing we need to notice is that the scaling vector field is not compatible with massive Dirac equations and Klein-Gordon equations; for more details, one can see (\ref{commutators}) and (\ref{commutatorsL0}). Therefore, we can not use the scaling vector field when applying the vector field method. Due to the absence of the scaling vector field, we can only use the global Sobolev inequality (inequality~(\ref{est:GloSob})) that does not involve the scaling vector field. To apply inequality~(\ref{est:GloSob}), we need to control the solution in $L^2.$ However, it is hard to get good estimates for solutions which are independent of the mass parameters.
	What we care about is how to control the $L^2$ norm $\|\Gamma^If\|_{L^2_x}$ till time $2t$, which allows us to get the pointwise decay rate of $f$ at time $t$ when applying this inequality and this is the second difficulty we face (see Remark~\ref{remarkSob} for more details). Therefore, what we need to consider is to find a uniform way to handle the system with different mass parameters.

	\subsection{Key points of the proof}
	
	In this subsection, we present the key ideas of the proof. Due to the problems mentioned above, we apply the Banach fixed point theorem to establish global existence of the system (\ref{equ:DKG})-(\ref{equ:initial}). 
	Specifically, due to the utilization of Klainerman-Sobolev inequality (\ref{est:GloSob}), we rely on the fixed point iteration method rather than a bootstrap argument to prove global existence (see Remark~\ref{remarkSob}). From the inequality (\ref{est:GloSob}), we have to control the $L^2$ norm in time from 0 to $2t$ to get the decay rate at time $t$. In the standard Klainerman-Sobolev inequality, we only need to control the $L^2$ norm at time $t$.
	
	Additionally, we expect to find some new uniform $L^2$ estimates of undifferentiated solutions to Dirac equations, wave equations and Klein-Gordon equations with different mass parameters; see (\ref{est:uL2}) for more details. Combining with inequalities (\ref{est:phiL2}) and (\ref{est:uL22}), we can treat the system (\ref{equ:DKG}) which has different mass parameters. Then we can get some $L^\infty$ estimates by using Klainerman-Sobolev inequality (\ref{est:GloSob}) which yields a fast $\langle t\rangle^{-\frac{3}{2}}$ decay rate. 
	
	Finally, to apply the iteration method, we need to construct a suitable normed function space and a contraction mapping. Then, by the Banach fixed point theorem and those estimates, we can complete the proof of Theorem~\ref{thm:main}. See Section~\ref{existence} for more details.

	\subsection{Organization of the paper}
	The paper is organized as follows. First, in Section \ref{Se:Pre}, we introduce the notation and basic estimates for later use. Then we give some $L^2$ estimates for solutions to 4D linear Dirac equations and 4D linear Klein-Gordon equations in Section \ref{estimates}. Finally, we establish global existence of the system (\ref{equ:DKG})-\eqref{equ:initial} in Section \ref{existence}.

	\section{Preliminaries}\label{Se:Pre}
	
	\subsection{Notation and conventions}
	Our problem is in (1+4) dimensional spacetime $\R^{1+4}$. We denote a spacetime point in $\R^{1+4}$ by 
	$(t,x)=(x_{0},x_{1},x_{2},x_{3},x_{4})$, and its spatial radius by $r=\sqrt{x_{1}^{2}+x_{2}^{2}+x_{3}^{2}+x_{4}^{2}}$. 
	Greek letters $\left\{\alpha,\beta,\dots\right\}$ are used to denote spacetime indices and Roman letters $\left\{a,b,\dots\right\}$ are used to denote spatial indices. We use Japanese bracket to denote $\langle \rho\rangle=(1+\rho^{2})^{\frac{1}{2}}$ for $\rho\in \R$.
	We write $A \lesssim B$ to indicate $A\leq C_0 B$ with $C_0$ a universal constant. Unless otherwise specified we will always adopt the Einstein summation convention for repeated upper and lower indices.
	
	Following Klainerman~\cite{Klainerman86}, we introduce the following vector fields
	\begin{enumerate}
		\item [(1)] Translations: $\partial_{\alpha}:=\partial_{x_{\alpha}}$, for $\alpha=0,1,2,3,4$.
		\item [(2)] Rotations: $\Omega_{ab}:=x_{a}\partial_{b}-x_{b}\partial_{a}$, for $1\leq a<b\leq 4$.
		\item [(3)] Scaling vector field: $L_0:=t\pt +x^{a}\partial_{a}$.
		\item [(4)] Lorentz boosts: $L_{a}:=t\partial_{a}+x_{a}\partial_{t}$, for $a=1,2,3,4$.
	\end{enumerate}
	Following Bachelot~\cite{Bache}, we also introduce the modified rotations and Lorentz boosts
	\begin{equation*}
		\widehat{\Omega}_{ab}=\Omega_{ab}-\frac{1}{2}\gamma^{a}\gamma^{b}\quad \mbox{and}\quad 
		\widehat{L}_{a}=L_{a}-\frac{1}{2}\gamma^{0}\gamma^{a}.
	\end{equation*}
	We denote
	\begin{equation*}
		\begin{aligned}
			\widehat{\Omega}=(\widehat{\Omega}_{12},\widehat{\Omega}_{13},\widehat{{\Omega}}_{14},\widehat{\Omega}_{23},\widehat{{\Omega}}_{24},\widehat{{\Omega}}_{34}),\quad
			\widehat{L}=(\widehat{L}_{1},\widehat{L}_{2},\widehat{L}_{3},\widehat{L}_{4}).
		\end{aligned}
	\end{equation*}
	We define two ordered sets of vector fields,
	\begin{equation*}
		\begin{aligned}
			&\Gamma=\left(\Gamma_{1},\dots,\Gamma_{15}\right)=\left(\partial,\Omega,L\right),\quad \ \ \ \
			\widehat{\Gamma}=\left(\widehat{\Gamma}_{1},\dots,\widehat{\Gamma}_{15}\right)=\left(\partial,\widehat{\Omega},\widehat{L}\right),\\
		\end{aligned}
	\end{equation*}
	where 
	\begin{equation*}
		\begin{aligned}
			\partial&=(\partial_0,\partial_1,\partial_2,\partial_3,\partial_4)=(\partial_t,\nabla),\\ \Omega&=(\Omega_{12},\Omega_{13},\Omega_{14},\Omega_{23},\Omega_{24},\Omega_{34}), \\
			L&=(L_1,L_2,L_3,L_4).
		\end{aligned}
	\end{equation*}
	Moreover, for all multi-index $I=(I_{1},\dots,I_{15})\in \mathbb{N}^{15}$, we denote
	\begin{equation*}
		\Gamma^{I}=\prod_{k=1}^{15}\Gamma_{k}^{I_{k}},\quad 
		\widehat{\Gamma}^{I}=\prod_{k=1}^{15}\widehat{\Gamma}_{k}^{I_{k}}.
	\end{equation*}
	Also, for any $\R$-valued or $\C^{4}$-valued function $f$, we have 
	\begin{equation*}
		\begin{aligned}
			\left|\Gamma f\right|=\left(\sum_{k=1}^{15}\left|\Gamma_{k}f\right|^{2}\right)^{\frac{1}{2}}\quad \mbox{and}\quad \big|\widehat{{\Gamma}}f\big|=\left(\sum_{k=1}^{15}\big|\widehat{\Gamma}_{k}f\big|^{2}\right)^{\frac{1}{2}}.
		\end{aligned}
	\end{equation*}
	Finally, let us review the Fourier transform 
	$$ \mathcal{F}(f)(\xi)=\hat{f}(\xi)=\int_{\R^{4}}f(x)e^{-ix\cdot\xi}\d x, \quad \mbox{for} \ f\in L^2_x.$$
	
	\subsection{Preliminary estimates}
	In this subsection, we recall several preliminary estimates that are related to the vector fields. First of all, from the definitions of $\Gamma$ and $\widehat{{\Gamma}}$, we can easily see that the differences between them are constant matrices. Then we have 
	\begin{equation}\label{est:hatGa}
		\sum_{|I|\le K}\big|\widehat{{\Gamma}}^{I}f\big|\lesssim \sum_{|I|\le K}\big|{\Gamma}^{I}f\big|\lesssim \sum_{|I|\le K}\big|\widehat{{\Gamma}}^{I}f\big|,
	\end{equation}
	\begin{equation}\label{est:hatparGa}
		\sum_{|I|\le K}\big|\partial\widehat{{\Gamma}}^{I}f\big|\lesssim \sum_{|I|\le K}\big|\partial{\Gamma}^{I}f\big|\lesssim \sum_{|I|\le K}\big|\partial\widehat{{\Gamma}}^{I}f\big|,
	\end{equation}
	for any smooth $\mathbb{R}$-valued or $\mathbb{C}^{4}$-valued function $f$ and $K\in \mathbb{N}^{+}$.
	
	Now we recall the following estimates related to the vector fields.
	
	\begin{lemma}\cite{Sogge}
		We have the following identities related to commutators
		\begin{equation}\label{commutators}
			[-\Box,\Gamma_k]=0,\quad [-i\gamma^{\mu}\partial_{\mu},\widehat{\Gamma}_{k}]=0,
		\end{equation}
		\begin{equation}\label{commutatorsL0}
			[-\Box,L_0]=-2\Box,\quad [-i\gamma^{\mu}\partial_{\mu},L_0]=-i\gamma^{\mu}\partial_\mu.
		\end{equation}
	\end{lemma}
	\begin{proof}
		Based on the definitions of the vector fields and an elementary computation, we can obtain~(\ref{commutators}) and (\ref{commutatorsL0}).
	\end{proof}
	
	\begin{lemma}
		For any multi-index $I\in \mathbb{N}^{15}$, smooth $\mathbb{R}$-valued function $f$ and $\mathbb{C}^{4}$-valued function $\Phi$, we have 
		\begin{equation}\label{est:hatGfPhi}
			\big|\widehat{{\Gamma}}^{I}(fF\Phi)\big|\lesssim \sum_{|I_{1}|+|I_{2}|\le |I|}\left|\Gamma^{I_{1}}f\right|\big|\widehat{{\Gamma}}^{I_{2}}\Phi\big|,
		\end{equation}
		where $F$ is a matrix of size $4\times4$ satisfying~(\ref{FH}).
	\end{lemma}
	
	\begin{proof}
		By an elementary computation and the Leibniz rule, we have
		\begin{equation*}
			\partial_\alpha(fF\Phi)=(\partial_\alpha f)F\Phi+fF(\partial_\alpha\Phi),\\
		\end{equation*} 
		\begin{equation*}
			\begin{aligned}
				\widehat{{\Omega}}_{ab}(fF\Phi)&=(\Omega_{ab}-\frac{1}{2}\gamma^a\gamma^b)(fF\Phi)
				=\Omega_{ab}(fF\Phi)-f(\frac{1}{2}\gamma^a\gamma^bF\Phi)\\
				&=(\Omega_{ab}f)F\Phi+fF(\Omega_{ab}\Phi)-f(\frac{1}{2}\gamma^a\gamma^bF)\Phi\\
				&=(\Omega_{ab}f)F\Phi+fF(\widehat{\Omega}_{ab}\Phi)+\frac{1}{2}f(F\gamma^a\gamma^b-\gamma^a\gamma^bF)\Phi,\\
			\end{aligned}
		\end{equation*}
		\begin{equation*}
			\begin{aligned}
				\widehat{L}_a(fF\Phi)&=(L_a-\frac{1}{2}\gamma^0\gamma^a)(fF\Phi)=
				L_{a}(fF\Phi)-f(\frac{1}{2}\gamma^0\gamma^aF\Phi)\\
				&=(L_{a}f)F\Phi+fF(L_{a}\Phi)-f(\frac{1}{2}\gamma^0\gamma^aF)\Phi\\
				&=(L_{a}f)F\Phi+fF(\widehat{L}_{a}\Phi)+\frac{1}{2}f(F\gamma^0\gamma^a-\gamma^0\gamma^aF)\Phi.
			\end{aligned}
		\end{equation*}
		By the above identities, we deduce~\eqref{est:hatGfPhi} for $I\in\mathbb{N}^{15}$ with $|I|=1.$
		
		Then using an induction argument, we obtain~\eqref{est:hatGfPhi}.
	\end{proof}
	
	\begin{lemma}
		For any multi-index $I\in \mathbb{N}^{15}$, and any smooth $\mathbb{C}^{4}$-valued functions $\Phi_{1}$ and $\Phi_{2}$, we have 
		\begin{equation}\label{est:hatGG}
			\left|\Gamma^{I}\left(\Phi_{1}^{*}H\Phi_{2}\right)\right|\lesssim \sum_{|I_{1}|+|I_{2}|\le |I|}\big|\widehat{{\Gamma}}^{I_{1}}\Phi_{1}\big|\big|\widehat{{\Gamma}}^{I_{2}}\Phi_{2}\big|,
		\end{equation}
		where $H$ is a matrix of size $4\times4$ satisfying~(\ref{FH}).
	\end{lemma}
	
	\begin{proof}
		By~\eqref{equ:gamma} and an elementary computation, we have 
		\begin{equation*}
			\begin{aligned}
				\Omega_{ab}\left(\Phi_{1}^{*}\right)&=\left(\Omega_{ab}\Phi_{1}\right)^*=\big(\widehat{\Omega}_{ab}\Phi_{1}\big)^{*}-\frac{1}{2}\Phi_{1}^{*}\gamma^{a}\gamma^{b},\\ L_{a}\left(\Phi_{1}^{*}\right)&=\left(L_a\Phi_{1}\right)^*=\big(\widehat{L}_{a}\Phi_{1}\big)^{*}+\frac{1}{2}\Phi_{1}^{*}\gamma^{0}\gamma^{a}.
			\end{aligned}
		\end{equation*}
		Then by the Leibniz rule of $\partial$, we obtain
		\begin{equation*}
			\begin{aligned}
				\partial_\alpha\left(\Phi_{1}^*H\Phi_{2}\right)&=\big(\partial_{\alpha}\left(\Phi_{1}^{*}\right)\big)H\Phi_{2}+\Phi_{1}^*H\big(\partial_{\alpha}\Phi_{2}\big)=\big(\partial_{\alpha}\Phi_{1}\big)^*H\Phi_{2}+\Phi_{1}^*H\big(\partial_{\alpha}\Phi_{2}\big),\\
				\Omega_{ab}\left(\Phi_{1}^{*}H\Phi_{2}\right)&=\big(\Omega_{ab}\left(\Phi_{1}^{*}\right)\big)H\Phi_{2}+\Phi_{1}^*H\big(\Omega_{ab}\Phi_{2}\big)\\
				&=\big(\widehat{\Omega}_{ab}\Phi_{1}\big)^{*}H\Phi_{2}+\Phi_{1}^{*}H\big({\widehat{\Omega}_{ab}}\Phi_{2}\big)-\frac{1}{2}\Phi_{1}^{*}\left(\gamma^{a}\gamma^{b}H-H\gamma^{a}\gamma^{b}\right)\Phi_{2},\\
				L_{a}\left(\Phi_{1}^{*}H\Phi_{2}\right)&=\big(L_{a}\left(\Phi_{1}^{*}\right)\big)H\Phi_{2}+\Phi_{1}^*H\big(L_{a}\Phi_{2}\big)\\
				&=\big(\widehat{L}_{a}\Phi_{1}\big)^{*}H\Phi_{2}+\Phi_{1}^{*}H\big({\widehat{L}_{a}}\Phi_{2}\big)+\frac{1}{2}\Phi_{1}^{*}\left(\gamma^{0}\gamma^{a}H+H\gamma^{0}\gamma^{a}\right)\Phi_{2}.\\
			\end{aligned}
		\end{equation*}
		Combining the above identities with the Leibniz rule of $\partial$ and then using an induction argument, we can get~\eqref{est:hatGG}.
	\end{proof}
	
	The following lemma can be found in~\cite{Klainerman85}.
	\begin{lemma}\label{le:GlobalSobolev}
		For all sufficiently regular functions $f=f(t,x)$ and all $I\in\mathbb{N}^{15},$ we have the following Klainerman-Sobolev inequality
		\begin{equation}\label{est:GloSob}
			|f(t,x)|\lesssim \langle t\rangle^{-\frac{3}{2}}\sup_{0\leq s\leq 2t}\sum_{|I|\le 4}\left\|\Gamma^{I}f(s)\right\|_{L^2_x}.
		\end{equation}
	\end{lemma}
	
	\begin{remark}\label{remarkSob}
		Based on Klainerman-Sobolev inequality (\ref{est:GloSob}), we can obtain the pointwise decay estimate of a function with $\langle t\rangle^{-\frac{3}{2}}$ rate, which is very useful to establish global existence of solutions to the system~(\ref{equ:DKG}). We do not need to bound the $L^2$ norm involving the scaling vector field $L_0$ (such as $\|L_0\Gamma^I f\|_{L^2_x}$), which is not compatible with the Dirac-Klein-Gordon system. What we need to consider is how to control the $L^2$ norm $\|\Gamma^If\|_{L^2_x}$ till time $2t$ to get the pointwise decay rate of $f$ at time $t$ when applying the inequality (\ref{est:GloSob}). Thus we will use the fixed point iteration method to complete the proof of Theorem~\ref{thm:main}.
	\end{remark}

	\section{$L^2$ estimates for linear equations}\label{estimates}
	In this section, we will give some $L^2$ estimates for Dirac equations and Klein-Gordon equations. These estimates will help us bound the $L^2$ norm of the undifferentiated solutions to Dirac equations and Klein-Gordon equations.
	
	First, we introduce the $L^2$ estimate for the solution to the Dirac equation.
	\begin{lemma}\label{lem:solL2Dirac}
		Suppose $\phi=\phi(t,x)$ is the solution to the Cauchy problem 
		\begin{equation*}
			-i\gamma^{\mu}\partial_{\mu}\phi+M\phi=G(t,x)\quad \mbox{with}\quad \phi|_{t=t_0}=\phi_{0},
		\end{equation*}
		where $G(t,x):\R^{1+4}\to \C^{4}$ is a sufficiently regular function, 
		then we have 
		\begin{equation}\label{est:phiL2}
			\|\phi\|_{L^2_x}\leq\|\phi_{0}\|_{L^2_x}+\int_{t_0}^{t}\|G(s,x)\|_{L^2_x}\d s.
		\end{equation}
	\end{lemma}
	
	\begin{proof}
		Multiplying the equation by $i\phi^*\gamma^0$, then we get
		\begin{equation}\label{multiDirac}
			\phi^*\partial_t\phi+\phi^*\gamma^0\gamma^a\partial_a\phi+iM\phi^*\gamma^0\phi=i\phi^*\gamma^0G.
		\end{equation}
		Taking the conjugate transpose of~(\refeq{multiDirac}), we have
		\begin{equation}\label{Hermitian}
			(\partial_t\phi)^*\phi+(\partial_a\phi)^*\gamma^0\gamma^a\phi-iM\phi^*\gamma^0\phi=-iG^*\gamma^0\phi.
		\end{equation}
		Adding (\ref{Hermitian}) and (\refeq{multiDirac}), we obtain
		\begin{equation*}
			\partial_t(\phi^*\phi)+\partial_a(\phi^*\gamma^0\gamma^a\phi)=-2{\rm{Im}}(\phi^*\gamma^0G).
		\end{equation*}
		Integrating with respect to $x$, 
		\begin{equation*}
			\partial_t\|\phi\|_{L^2_x}^2\leq2\int_{\R^{4}}|\phi^*\gamma^0G|\d x\leq2\|\phi\|_{L^2_x}\|G\|_{L^2_x}.
		\end{equation*}
		Then we obtain
		\begin{equation*}
			\partial_t\|\phi\|_{L^2_x}\leq\|G\|_{L^2_x}.
		\end{equation*}
		Finally integrating with respect to $t$, we get (\ref{est:phiL2}).
	\end{proof}
	Now, we give some estimates for the solution to the Klein-Gordon equation.
	\begin{lemma}\label{lem:solKG}
		Suppose $u=u(t,x)$ is the solution to the Cauchy problem
		\begin{equation}\label{equ: masswave}
			-\Box u+m^2u=G(t,x)   \quad   \mbox{with}  \quad (u,\partial_tu)|_{t=t_0}=(u_0,u_1),
		\end{equation}
		where $G(t,x):\R^{1+4}\to \R$ is a sufficiently regular function, 
		then the following estimates hold.
		\begin{enumerate}
			\item {If \rm{$m\in[0,1)$,}} for all $\delta_1\in\R$, we have
			\begin{equation}\label{est:uL2}
				\|u\|_{L^2_x}\lesssim\|u_0\|_{L^2_x}+\|u_1\|_{L^1_x}+\|u_1\|_{L^2_x}+\int_{t_0}^{t}(s^{\delta_1}\|G(s,x)\|_{L^2_x}+s^{-\delta_1}\|G(s,x)\|_{L^1_x})\d s.
			\end{equation}
			\item {If \rm {$m=1$,}} it holds
			\begin{equation}\label{est:uL22}
				\|u\|_{L^2_x}\lesssim\|u_0\|_{L^2_x}+\|u_1\|_{L^2_x}+\int_{t_0}^{t}\|G(s,x)\|_{L^2_x} \d s.
			\end{equation}
			\item\cite{Sogge} {If \rm{$m\in[0,1]$,}} there holds
			\begin{equation}\label{est:energy}
				\mathcal{E}_m(t,u)^{\frac{1}{2}}\lesssim\mathcal{E}_m(t_0,u)^{\frac{1}{2}}+\int_{t_0}^{t}\|G(s,x)\|_{L^2_x}\d s,
			\end{equation}
			where $ \mathcal{E}_m(t,u)$ is the natural energy defined by
			\begin{equation*}
				\mathcal{E}_m(t,u):=\int_{\R^{4}}(|\partial_tu|^2+\sum_{a=1,2,3,4}|\partial_au|^2+m^2u^2)\d x.
			\end{equation*}
		\end{enumerate}
	\end{lemma}
	
	\begin{proof}
		\textbf{Proof of (i).} Suppose that $m\in[0,1).$ Taking the Fourier transform in the Cauchy problem \eqref{equ: masswave} with respect to $x$ and using the Duhamel's principle, we get
		\begin{equation*}
			\left\{\begin{aligned}
				\partial_{tt} \hat{u}(t,\xi)+|\xi|^2\hat{u}(t,\xi)+m^2\hat{u}(t,\xi)=\widehat{G}(t,\xi),\\
				\hat{u}(t_0,\xi)=\hat{u}_0(\xi),\quad \partial_t\hat{u}(t_0,\xi)=\hat{u}_1(\xi).
			\end{aligned}\right.
		\end{equation*}
		We solve the above second order ODE with respect to $t$ obtaining the expression of the solution $u$ in Fourier space
		\begin{equation*}
			\begin{aligned}
				\hat{u}(t,\xi)=\cos(\xi_m(t-t_0))\hat{u}_0(\xi)+\frac{\sin(\xi_m(t-t_0))}{\xi_m}\hat{u}_1(\xi)
				+\int_{t_0}^{t}\frac{\sin(\xi_m(t-s))}{\xi_m}\widehat{G}(s,\xi)\d s,
			\end{aligned}
		\end{equation*}
		for $(t,\xi)\in [t_0,\infty)\times\mathbb{R}^4$, where $\xi_m=\sqrt{|\xi|^2+m^2}.$
		
		Then from the Plancherel theorem, we deduce
		\begin{equation*}
			\begin{aligned}
				\|u\|_{L^2_x}\lesssim\|\hat{u}(t,\xi)\|_{L^2_\xi}
				&\lesssim\left\|\cos(\xi_m(t-t_0))\hat{u}_0(\xi)\right\|_{L^2_\xi}+\left\|\frac{\sin(\xi_m(t-t_0))}{\xi_m}\hat{u}_1(\xi)\right\|_{L^2_\xi}\\
				&+\int_{t_0}^{t}\left\|\frac{\sin(\xi_m(t-s))}{\xi_m}\widehat{G}(s,\xi)\right\|_{L^2_\xi}\d s.
			\end{aligned}
		\end{equation*}
		Now, we estimate each term of the above inequality in turn. For the first one, we have
		\begin{equation*}
			\left\|\cos(\xi_m(t-t_0))\hat{u}_0(\xi)\right\|_{L^2_\xi}\lesssim\|\hat{u}_0(\xi)\|_{L^2_\xi}.
		\end{equation*}
		Then, for the second one, we can deduce
		\begin{equation*}
			\begin{aligned}
				&\left\|\frac{\sin(\xi_m(t-t_0))}{\xi_m}\hat{u}_1(\xi)\right\|_{L^2_\xi}\\
				\lesssim&\left\|\frac{\sin(\xi_m(t-t_0))}{\xi_m}\hat{u}_1(\xi)\textbf{1}_{\xi_m\geq1}\right\|_{L^2_\xi}+\left\|\frac{\sin(\xi_m(t-t_0))}{\xi_m}\hat{u}_1(\xi)\textbf{1}_{\xi_m\leq1}\right\|_{L^2_\xi}\\
				\lesssim&\|\hat{u}_1(\xi)\|_{L^2_\xi}+\left(\int_{|\omega|=1}\int_{0}^{\sqrt{1-m^2}}r\hat{u}^2_1(r) \d r \d \omega\right)^\frac{1}{2}\\
				\lesssim&\|\hat{u}_1(\xi)\|_{L^2_\xi}+\|\hat{u}_1(\xi)\|_{L^\infty_\xi}.
			\end{aligned}
		\end{equation*}
		Similarly, for the third one, we get
		\begin{equation*}
			\begin{aligned}
				&\left\|\frac{\sin(\xi_m(t-s))}{\xi_m}\widehat{G}(s,\xi)\right\|_{L^2_\xi}\\
				\lesssim &\left\|\frac{\sin(\xi_m(t-s))}{\xi_m}\widehat{G}(s,\xi)\textbf{1}_{|\xi|\geq s^{-\delta_1}}\right\|_{L^2_\xi}+\left\|\frac{\sin(\xi_m(t-s))}{\xi_m}\widehat{G}(s,\xi)\textbf{1}_{|\xi|\leq s^{-\delta_1}}\right\|_{L^2_\xi}\\
				\lesssim &s^{\delta_1}\|\widehat{G}(s,\xi)\|_{L^2_\xi}+\left(\int_{|\omega|=1}\int_{0}^{s^{-\delta_1}}r\widehat{G}^2(s,r) \d r \d \omega\right)^\frac{1}{2}\\
				\lesssim&s^{\delta_1}\|\widehat{G}(s,\xi)\|_{L^2_\xi}+s^{-\delta_1}\|\widehat{G}(s,\xi)\|_{L^\infty_\xi},\quad\mbox{for all}\ \delta_1\in\R.
			\end{aligned}
		\end{equation*}
		Therefore, using again the Plancherel theorem and the above inequalities, we obtain (\ref{est:uL2}).
		
		\textbf{Proof of (ii).} Let $m=1.$ We obtain the expression of the solution $u$ in Fourier space
		\begin{equation*}
			\begin{aligned}
				\hat{u}(t,\xi)=&\cos(\sqrt{|\xi|^2+1}(t-t_0))\hat{u}_0(\xi)+\frac{\sin(\sqrt{|\xi|^2+1}(t-t_0))}{\sqrt{|\xi|^2+1}}\hat{u}_1(\xi)\\
				&+\int_{t_0}^{t}\frac{\sin(\sqrt{|\xi|^2+1}(t-s))}{\sqrt{|\xi|^2+1}}\widehat{G}(s,\xi)\d s,
			\end{aligned}
		\end{equation*}
		for $(t,\xi)\in [t_0,\infty)\times\mathbb{R}^4$.
		
		Also,
		\begin{equation*}
			\begin{aligned}
				&\left\|\cos(\sqrt{|\xi|^2+1}(t-t_0))\hat{u}_0(\xi)\right\|_{L^2_\xi}\lesssim\|\hat{u}_0(\xi)\|_{L^2_\xi},\\
				&\left\|\frac{\sin(\sqrt{|\xi|^2+1}(t-t_0))}{\sqrt{|\xi|^2+1}}\hat{u}_1(\xi)\right\|_{L^2_\xi}\lesssim\|\hat{u}_1(\xi)\|_{L^2_\xi},\\
				&\left\|\frac{\sin(\sqrt{|\xi|^2+1}(t-s))}{\sqrt{|\xi|^2+1}}\widehat{G}(s,\xi)\right\|_{L^2_\xi}\lesssim\|\widehat{G}(s,\xi)\|_{L^2_\xi}.\\
			\end{aligned}
		\end{equation*}
		Therefore, using again the Plancherel theorem and the above inequalities, we can obtain (\ref{est:uL22}).
		
		\textbf{Proof of (iii).}
		The proof is standard, and we omit it.
	\end{proof}

	\section{Proof of Theorem~\ref{thm:main}}\label{existence}
	In this section, we will rely on the Banach fixed point theorem to prove global existence of the system (\ref{equ:DKG}). Remember that $N\geq 7$ in Theorem~\ref{thm:main}.
	\subsection{Setting of the iteration method}
	First, we introduce the foundations for the fixed point iteration method: a function space and a contraction mapping. The function space represented by $X$ is defined as follows.
	\begin{definition}\label{normspace}
		Let $\phi=\phi(t,x), u=u(t,x)$ be sufficiently regular functions. We say that a pair $(\phi, u)$ belongs to the function space $X$ if the pair satisfies
		\begin{enumerate}
			\item[(i)] $(\phi,u,\partial_t u)|_{t=0}=(\psi_{0},v_0,v_1),$ 
			\item[(ii)] $\|(\phi,u)\|_X\leq C_1\varepsilon$,
		\end{enumerate}
		in which $C_1\gg1$ is a large constant to be chosen, the size of the initial data $\varepsilon\ll1$ is sufficiently small such that $C_1\varepsilon\ll1$, and the norm $\|(\cdot,\cdot)\|_X$ is defined by
		\begin{equation}\label{norm}
			\|(\phi,u)\|_X:=\sup_{t\geq0,|I|\leq N}\|\widehat{\Gamma}^I\phi\|_{L^2_x}+\sup_{t\geq0,|I|\leq N}\langle t\rangle^{-\frac{1}{4}}\|\Gamma^Iu\|_{L^2_x}.
		\end{equation}
	\end{definition}
	
	\begin{remark}
		Note that the space $X$ is nonempty. For example, let $\phi=\phi(t,x)$ be the solution to the 4D homogeneous Dirac equation 
		\begin{equation*}
			-i\gamma^{\mu}\partial_{\mu} \phi=0   \quad   \mbox{with}  \quad   \phi|_{t=0}=\psi_{0}.
		\end{equation*}
		Similarly, let $u=u(t,x)$ be the solution to the 4D homogeneous wave equation
		\begin{equation*}
			-\Box u=0   \quad   \mbox{with}  \quad (u,\partial_tu)|_{t=0}=(v_0,v_1).
		\end{equation*}
		Then $(\phi,u)$ belongs to $X$ by (\ref{est:smallness}) and the constructions of $\phi,u$.
	\end{remark}
	
	\begin{remark}
		In (\ref{norm}), the weight $\langle t\rangle^{-\frac{1}{4}}$ is chosen according to Lemma~\ref{lem:solKG} (i). When taking different $\delta_1$ in~(\ref{est:uL2}), we have different weights. Since the decay rates of $\|\Gamma^I(\phi^*H\phi)\|_{L^2_x}$ and $\|\Gamma^I(\phi^*H\phi)\|_{L^1_x}$ are different (see Lemma~\ref{est:nonlinear}), we need to take suitable $\delta_1$ to balance these two decay rates. The weight $\langle t\rangle^{-\frac{1}{4}}$ is optimal in our method.
	\end{remark}
	We note that the function space $X$ is complete with respect to the metric induced by the norm $\|(\cdot,\cdot)\|_X$.
	
	Next, we give the definition of the solution mapping $T$.
	\begin{definition}\label{mappingT}
		For any $(\phi,u)\in X$, define a mapping T by
		\begin{equation}\label{defT}
			T(\phi,u)=(\tilde{\phi},\tilde{u}),
		\end{equation}
		where $(\tilde{\phi},\tilde{u})$ is the solution to the following Cauchy problem
		\begin{equation}\label{defequ}
			\left\{\begin{aligned}
				-i\gamma^{\mu}\partial_{\mu}\tilde{\phi}+M\tilde{\phi}&=uF\phi,\\
				-\Box \tilde{u}+m^2\tilde{u}&=\phi^{*}H\phi,\\
				(\tilde{\phi},\tilde{u},\partial_t \tilde{u})|_{t=0}&=(\psi_{0},v_0,v_1).
			\end{aligned}\right.
		\end{equation}
	\end{definition}

	\subsection{Global existence of the coupled system (\ref{equ:DKG})}
	
	In this subsection, we will give the proof of Theorem~\ref{thm:main}. Before that,  let us give some important estimates.
	\begin{lemma}\label{highpointdecay}
		For $I\in\mathbb{N}^{15}$ with $|I|\leq N-4$, and $(\phi,u)\in X,$ we have
		$$ |\widehat{\Gamma}^I\phi(t,x)|\lesssim C_1\varepsilon\langle t\rangle^{-\frac{3}{2}},$$
		$$|\Gamma^Iu(t,x)|\lesssim C_1\varepsilon\langle t\rangle^{-\frac{5}{4}}.$$
	\end{lemma}
	\begin{proof}
		According to Klainerman-Sobolev inequality~(\ref{est:GloSob}), (\ref{est:hatGa}) and Definition~\ref{normspace}, for $|I|\leq N-4$, we have
		\begin{equation*}
			\begin{aligned}
				|\widehat{\Gamma}^I\phi(t,x)|&\lesssim\langle t\rangle^{-\frac{3}{2}}\sup_{\substack{0\leq s\leq 2t\\|I|\leq N-4\\|J|\leq 4}}\|\Gamma^J\Gamma^I\phi(s)\|_{L^2_x}\\
				&\lesssim\langle t\rangle^{-\frac{3}{2}}\sup_{\substack{0\leq s\leq 2t\\|L|\leq N}}\|\Gamma^L\phi(s)\|_{L^2_x}\\
				&\lesssim\langle t\rangle^{-\frac{3}{2}}\sup_{\substack{0\leq s\leq 2t\\|L|\leq N}}\|\widehat{\Gamma}^L\phi(s)\|_{L^2_x}
				\lesssim C_1\varepsilon\langle t\rangle^{-\frac{3}{2}},
			\end{aligned}
		\end{equation*}
		where $I,J,L\in\mathbb{N}^{15}$.
		
		Similarly, for $|I|\leq N-4$, we can obtain
		\begin{equation*}
			|\Gamma^Iu(t,x)|\lesssim\langle t\rangle^{-\frac{3}{2}}\sup_{\substack{0\leq s\leq 2t\\|I|\leq N-4\\|J|\leq 4}}\|\Gamma^J\Gamma^Iu(s)\|_{L^2_x}\lesssim C_1\varepsilon\langle t\rangle^{-\frac{5}{4}},
		\end{equation*}
		where $I,J\in\mathbb{N}^{15}$.
	\end{proof}
	
	\begin{lemma}\label{est:nonlinear}
		Let $|I|\leq N$ with $N\geq7$. We have the following $L^2$ and $L^1$ estimates on the nonlinear terms. There hold
		\begin{enumerate}
			\item $\|\widehat{\Gamma}^I(uF\phi)\|_{L^2_x}\lesssim (C_1\varepsilon)^2\langle t\rangle^{-\frac{5}{4}},$
			\item $\|\Gamma^I(\phi^*H\phi)\|_{L^2_x}\lesssim (C_1\varepsilon)^2\langle t\rangle^{-\frac{3}{2}},$
			\item $\|\Gamma^I(\phi^*H\phi)\|_{L^1_x}\lesssim (C_1\varepsilon)^2.$
		\end{enumerate}
	\end{lemma}
	\begin{proof}
		\textbf{Proof of (i).} Let $|I|\leq N$ with $N\geq7.$ By (\ref{est:hatGfPhi}), Definition~\ref{normspace}, Lemma~\ref{highpointdecay} and H\"{o}lder inequality, we have
		\begin{equation*}
			\begin{aligned}
				&\|\widehat{\Gamma}^I(uF\phi)\|_{L^2_x}\\
				\lesssim&\sum_{\substack{|I_1|+|I_2|\leq N\\|I_1|\leq N,|I_2|\leq N-4}}\|\Gamma^{I_1}u\|_{L^2_x}\|\widehat{\Gamma}^{I_2}\phi\|_{L^\infty_x}+\sum_{\substack{|I_1|+|I_2|\leq N\\|I_1|\leq N-4,|I_2|\leq N}}\|\Gamma^{I_1}u\|_{L^\infty_x}\|\widehat{\Gamma}^{I_2}\phi\|_{L^2_x}\\
				\lesssim&(C_1\varepsilon)^2\langle t\rangle^{-\frac{5}{4}}.
			\end{aligned}
		\end{equation*}
		\textbf{Proof of (ii).} Let $|I|\leq N$ with $N\geq7.$ By (\ref{est:hatGG}), Definition~\ref{normspace}, Lemma~\ref{highpointdecay} and H\"{o}lder inequality, we have
		\begin{equation*}
			\begin{aligned}
				&\|\Gamma^I(\phi^*H\phi)\|_{L^2_x}\\
				\lesssim&\sum_{\substack{|I_1|+|I_2|\leq N\\|I_1|\leq N,|I_2|\leq N-4}}\|\widehat{{\Gamma}}^{I_1}\phi\|_{L^2_x}\|\widehat{\Gamma}^{I_2}\phi\|_{L^\infty_x}+\sum_{\substack{|I_1|+|I_2|\leq N\\|I_1|\leq N-4,|I_2|\leq N}}\|\widehat{{\Gamma}}^{I_1}\phi\|_{L^\infty_x}\|\widehat{\Gamma}^{I_2}\phi\|_{L^2_x}\\
				\lesssim&(C_1\varepsilon)^2\langle t\rangle^{-\frac{3}{2}}.
			\end{aligned}
		\end{equation*}
		\textbf{Proof of (iii).} Similarly to (ii), for $|I|\leq N$, we have
		\begin{equation*}
			\|\Gamma^I(\phi^*H\phi)\|_{L^1_x}
			\lesssim\sum_{\substack{|I_1|+|I_2|\leq N }}\|\widehat{{\Gamma}}^{I_1}\phi\|_{L^2_x}\|\widehat{\Gamma}^{I_2}\phi\|_{L^2_x}
			\lesssim(C_1\varepsilon)^2.
		\end{equation*}
	\end{proof}
	Now let us give some propositions about the function space $X$ and the solution mapping $T$ defined by Definitions~\ref{normspace} and \ref{mappingT}, respectively.
	\begin{proposition}\label{images}
		The image of the map $T$ is a subset of $X$.
	\end{proposition}
	\begin{proof}
		We need to show that for any $(\phi,u)\in X$, $T(\phi,u)=(\tilde{\phi},\tilde{u})\in X.$ 
		First, by the definition of the function space $X$, we have
		\begin{equation*}
			\|(\phi,u)\|_X=\sup_{t\geq0,|I|\leq N}\|\widehat{\Gamma}^I\phi\|_{L^2_x}+\sup_{t\geq0,|I|\leq N}\langle t\rangle^{-\frac{1}{4}}\|\Gamma^Iu\|_{L^2_x}\leq C_1\varepsilon.
		\end{equation*}
		Also, the inequality (\ref{est:smallness}) implies
		\begin{equation}\label{est:initial}
			\|\widehat{{\Gamma}}^I\tilde{\phi}(0,x)\|_{L^2_x}+\|\Gamma^I\tilde{u}(0,x)\|_{L^2_x}+\|\partial_t\Gamma^I\tilde{u}(0,x)\|_{L^1_x}+\|\partial_t\Gamma^I\tilde{u}(0,x)\|_{L^2_x}\lesssim\varepsilon,\quad\mbox{for} \ |I|\leq N.
		\end{equation}
		By Lemma~\ref{lem:solL2Dirac}, Definition~\ref{mappingT}, Lemma~\ref{est:nonlinear} and (\ref{est:initial}), we get
		\begin{equation}\label{est:highphi}
			\begin{aligned}
				\|\widehat{\Gamma}^I\tilde{\phi}\|_{L^2_x}&\lesssim\|\widehat{\Gamma}^I\tilde{\phi}(0,x)\|_{L^2_x}+\int_{0}^{t}\|\widehat{\Gamma}^I(uF\phi)(s,x)\|_{L^2_x}\d s\\
				&\lesssim \varepsilon+\int_{0}^{t}(C_1\varepsilon)^2\langle s\rangle^{-\frac{5}{4}}\d s\\
				&\lesssim\varepsilon+(C_1\varepsilon)^2.
			\end{aligned}
		\end{equation}
		When $m\in[0,1)$, based on Lemma~\ref{lem:solKG} (i) (taking $\delta_1=\frac{3}{4}$), Lemma~\ref{est:nonlinear} and (\ref{est:initial}), we deduce
		\begin{equation}\label{est:highum}
			\begin{aligned}
				\|\Gamma^I\tilde{u}\|_{L^2_x}&\lesssim\|\Gamma^I\tilde{u}(0,x)\|_{L^2_x}+\|\partial_t\Gamma^I\tilde{u}(0,x)\|_{L^1_x}+\|\partial_t\Gamma^I\tilde{u}(0,x)\|_{L^2_x}\\
				&+\int_{0}^{t}(s^{\frac{3}{4}}\|\Gamma^I(\phi^*H\phi)\|_{L^2_x}+s^{-\frac{3}{4}}\|\Gamma^I(\phi^*H\phi)\|_{L^1_x})\d s\\
				&\lesssim\varepsilon+(C_1\varepsilon)^2\int_{0}^{t}(s^{\frac{3}{4}}\langle s\rangle^{-\frac{3}{2}}+s^{-\frac{3}{4}})\d s\\
				&\lesssim\varepsilon+(C_1\varepsilon)^2\langle t\rangle^{\frac{1}{4}}.
			\end{aligned}
		\end{equation}
		When $m=1$, based on Lemma~\ref{lem:solKG} (ii), Lemma~\ref{est:nonlinear} and (\ref{est:initial}), we obtain
		\begin{equation}\label{est:highu}
			\begin{aligned}
				\|\Gamma^I\tilde{u}\|_{L^2_x}&\lesssim\|\Gamma^I\tilde{u}(0,x)\|_{L^2_x}+\|\partial_t\Gamma^I\tilde{u}(0,x)\|_{L^2_x}
				+\int_{0}^{t}\|\Gamma^I(\phi^*H\phi)\|_{L^2_x}\d s\\
				&\lesssim\varepsilon+(C_1\varepsilon)^2\int_{0}^{t}\langle s\rangle^{-\frac{3}{2}}\d s\\
				&\lesssim\varepsilon+(C_1\varepsilon)^2.
			\end{aligned}
		\end{equation}
		Thanks to the above inequalities~(\ref{est:highphi}), (\ref{est:highum}) and (\ref{est:highu}), we can choose $C_1\gg1$ large enough and $\varepsilon\ll1$ sufficiently small such that
		\begin{equation*}
			\|T(\phi,u)\|_X=\|(\tilde{\phi},\tilde{u})\|_X\leq\frac{1}{2}C_1\varepsilon.
		\end{equation*}
		Therefore, by Definitions~\ref{normspace} and \ref{mappingT}, the image of $T$ belongs to $X$.
	\end{proof}
	
	\begin{proposition}\label{contraction}
		The solution mapping $T$ defined by (\ref{defT}) and (\ref{defequ}) is a contraction mapping from $X$ to itself, i.e.,
		\begin{equation*}
			\|(\tilde{\phi}-\tilde{\phi'},\tilde{u}-\tilde{u'})\|_X\leq\frac{1}{2}\|(\phi-\phi',u-u')\|_X,
		\end{equation*}
		where $(\phi,u),(\phi',u')\in X$, and $(\tilde{\phi},\tilde{u})=T(\phi,u), (\tilde{\phi'},\tilde{u'})=T(\phi',u')$.
	\end{proposition}		
	\begin{proof}
		The proof of Proposition~\ref{contraction} is similar to the proof of Proposition~\ref{images}. We can shrink the size of the initial data if necessary, so we omit it.
	\end{proof}		
	
	With the above results, we can complete the proof of Theorem~\ref{thm:main}.
	\begin{proof}[\textbf{Proof of Theorem~\ref{thm:main}}]
		According to the Banach fixed point theorem, Propositions~\ref{images} and~\ref{contraction}, we know that the contraction mapping $T$ has a unique fixed point $(\psi,v)\in X$, which is the global-in-time solution to the coupled system (\ref{equ:DKG}). From Lemma~\ref{highpointdecay}, we can obtain pointwise decay estimates for the solution $(\psi,v)$
		\begin{equation}\label{est:psiv}
			|\psi(t,x)|\lesssim \varepsilon\langle t\rangle^{-\frac{3}{2}}, \quad |v(t,x)|\lesssim\varepsilon\langle t\rangle^{-\frac{5}{4}}.
		\end{equation}
		Also, from Lemma~\ref{equ: masswave} (iii), Lemma~\ref{est:nonlinear} and (\ref{est:smallness}), we have
		\begin{equation*}
			\begin{aligned}
				m\|v\|&\lesssim\mathcal{E}_m(0,v)^{\frac{1}{2}}+\int_{0}^{t}\|\psi^*H\psi\|\d s\\
				&\lesssim\|v_1\|+\|\nabla v_0\|+\|v_0\|+(C_1\varepsilon)^2\int_{0}^{t}\langle s\rangle^{-\frac{3}{2}}\d s\\
				&\lesssim\varepsilon+(C_1\varepsilon)^2.
			\end{aligned}
		\end{equation*}
		Thanks to Klainerman-Sobolev inequality~(\ref{est:GloSob}), we deduce
		\begin{equation}\label{est:mv}
			|v(t,x)|\lesssim\varepsilon\langle t\rangle^{-\frac{3}{2}}m^{-1}, \quad\mbox{for}\ m\in(0,1].
		\end{equation}
		By~(\ref{est:psiv}) and~(\ref{est:mv}), we can get unified decay estimates~(\ref{est:thmpoint}).
		This completes the proof of Theorem~\ref{thm:main}. 
	\end{proof}

	\section*{Declarations}
	\textbf{Conflict of interest}\quad
	The author declares that there is no conflict of interest regarding the publication of this paper.


\end{document}